\documentclass[abstracton]{scrartcl}
\usepackage[utf8]{inputenc}
\usepackage{amsmath,amssymb,amsthm}

\newtheorem{theorem}{Theorem}
\newtheorem{lemma}[theorem]{Lemma}

\newtheorem{observation}[theorem]{Observation}

\theoremstyle{remark}

\newcommand{\id}{\mathrm{id}}
\newcommand{\Sym}{\mathrm{Sym}}

\title{Asymmetric colouring of locally compact permutation groups}
\author{Florian Lehner}

\begin{document}

\maketitle

\begin{abstract}
     Let $G \leq \Sym (X)$ for a countable set $X$. 
     Call a colouring of $X$ asymmetric, if the identity is the only element of $G$ which preserves all colours. 
     The motion (also called minimal degree) of $G$ is the minimal number of elements moved by an element $g \in G \setminus\{\id\}$.
     We show that every locally compact permutation group with infinite motion admits an asymmetric $2$-colouring. This builds on, and generalises a recent result by Babai and confirms a conjecture by Imrich, Smith, Tucker, and Watkins from 2015.
\end{abstract}

\section{Introduction}

Let $X$ be a countable set and let $G \leq \Sym(X)$. A colouring of $X$ is called \emph{asymmetric}, if the identity is the only element of $G$ which preserves the colouring. The term asymmetric colouring was probably first introduced in \cite{babai-trees}, but the study of asymmetric colouring dates further back; for instance the proof of Frucht's theorem \cite{frucht} crucially relies on putting asymmetric `decorations' on the edges of a graph.

In 1996, Albertson and Collins reintroduced asymmetric colourings of graphs under the name \emph{distinguishing colourings} \cite{albertsoncollins}. Their paper sparked significant interest in the study of asymmetric graph colourings, leading to many nice results and conjectures, some of them by the author of the present paper \cite{babai-imc,ikt,istw,kwz,lehner-random,lehner-intermediate,lehner-edge,lps-bounded,russellsundaram,tucker-imc}. 

A notion that plays a role in many of these results and their proofs is the \emph{motion} (in permutation group theory more commonly known as the \emph{minimal degree}) of a permutation group. It is defined as the minimal number of elements moved by any $g \in G \setminus \{\id\}$. The connection between motion and asymmetric colourings was for instance made explicit in \cite{russellsundaram}, but it had been known much longer, see for instance \cite{cns,gluck}.

A conjecture by Tucker \cite{tucker-imc} which had been open for over a decade before it was recently settled by Babai \cite{babai-imc} asserts that motion and asymmetric colourings are also closely connected for infinite, locally finite graphs. 

\begin{theorem}[Babai \cite{babai-imc}]
If the automorphism group of a locally finite, connected graph has infinite motion on the vertex set $X$, then there is an asymmetric $2$-colouring of $X$. 
\end{theorem}

The main result of this short note generalises the above result to the setting of locally compact permutation groups, where topological properties are with respect to the permutation topology, see Section \ref{sec:notation}. 

\begin{theorem}
\label{thm:main}
Let $G \subseteq \Sym(X)$ be a locally compact permutation group. If $G$ has infinite motion, then there is an asymmetric $2$-colouring of $X$.
\end{theorem}

We point out that this result is not independent of \cite{babai-imc}; it relies on the main technical result of that paper which is paraphrased in Theorem \ref{thm:babai} below. 
Since every closed, subdegree finite permutation group is locally compact, Theorem \ref{thm:main}  confirms a generalisation of Tucker's conjecture which was stated by Imrich, Smith, Tucker, and Watkins in \cite{istw} and highlighted as an open problem in \cite{babai-imc}.

\section{Notation and preliminiary results}
\label{sec:notation}

Let $X$ be a countably infinite set. A \emph{permutation group} on $X$ is a group $G \leq \Sym(X)$ together with the induced action on $X$. As usual, for $Y \subseteq X$ we denote by $G_Y$ the \emph{setwise stabiliser}, and by $G_{(Y)}$ the \emph{pointwise stabiliser} of $Y$ in $G$. If $Y = \{x\}$ write $G_x$ instead of $G_{\{x\}}$.
A \emph{suborbit} of $G$ with respect to $x$ is an orbit of $G_x$. We call $G$ \emph{subdegree finite} if all suborbits are finite.

The group $\Sym (X)$ can be endowed with the topology of pointwise convergence on $X$ which in this context is often referred to as the \emph{permutation topology}. This topology is well studied, all results we need can for instance be found in \cite{cameron-topology}. A basis of this topology is given by sets of the form $\{g \in \Sym (X) \mid gx_i = y_i \text{ for }1 \leq i \leq k\}$, where $k \in \mathbb N$ and $x_1, \dots x_k, y_1, \dots y_k$ are elements of $X$. We say that a permutation group $G \leq \Sym (X)$ is \emph{closed}, \emph{compact}, or \emph{locally compact}, if it has the respective property as a subset of $\Sym (X)$ endowed with the permutation topology. It is known that a permutation group $G$  is compact, if and only if it is closed and all of its orbits are finite. Moreover, $G$ is locally compact if and only if it is closed and there is a finite set $Y \subseteq X$ such that all orbits of $G_Y$ are finite. Note that every closed, subdegree finite permutation group is locally compact because point stabilisers in such groups are compact, but the converse is not true. 

In what follows it will be useful to be able to interpret a group $G$ as a permutation group on different sets. The following observation tells us when this is possible.

\begin{observation}
\label{obs:smallerset}
Let $G$ be a permutation group on $X$ and $Y \subseteq X$ such that $GY = Y$ and for every $g \in G \setminus\{\id\}$ there is some $y \in Y$ such that $gy \neq y$. Then $G$ can be seen as a permutation group on $Y$.
\end{observation}


%

A \emph{$k$-colouring} of $X$ is a map $c \colon X \to C$ with $|C|=k$; throughout this short note we will only consider the case $k=2$ and let $C = \{0,1\}$. We call a $k$-colouring $c$ \emph{asymmetric} with respect to a subset $H \subseteq \Sym(X)$, if for every $g \in H \setminus\{id\}$ there is some $x \in X$ such that $c(x) \neq c(gx)$. Note that we do not require $H$ to be a subgroup.

The proof of Theorem \ref{thm:main} makes use of the following technical result, which is a rephrasing of \cite[Theorem 1.2]{babai-imc}.

\begin{theorem}
\label{thm:babai}
Let $U_i, i \in \mathbb N$ be disjoint finite sets. Let $G$ be a permutation group on $X = \bigcup_{i \in \mathbb N} U_i$ whose action setwise fixes each $U_i$.
If $G_{(U_{i+1})} \leq G_{(U_{i})}$ for every $i \in \mathbb N$, then there is an asymmetric $2$-colouring of $X$.
\end{theorem}

This result is highly non-trivial; its proof depends on the classification of finite simple groups and spans several pages.

\section{Proof of the main result}

Before proving our main result, we set up some notation which will be used throughout the proof. We fix a countable set $X$ and a locally compact permutation group $G$ on $X$ with infinite motion. Moreover, we fix a finite subset $X_0 \subseteq X$ such that $G_{X_0}$ only has finite orbits, and we let $(S_n)_{n \in \mathbb N}$ be an enumeration of the orbits of $G_{X_0}$.
For each $n \in \mathbb N$ we let 
\[G_n = \bigcap_{i \geq n} G_{S_i}.\]
Note that the set $X \setminus \bigcup_{i \geq n} S_i$ is finite and thus $G_n$ only has finite orbits. Moreover, $G_n$ is closed because the sets $\{g\in \Sym(X) \mid gx=y\}$ are open for all pairs $(x,y) \in X^2$, and the complement of  $G_n$ can be written as a union of $\Sym(X) \setminus G$ with countably many such sets. Thus $G_i$ is compact. 

Let $\mathcal X_n =\{g X_0 \mid g \in G_n\}$, and let $\mathcal X_{\infty} = \{Y \subseteq X \mid |Y| = |X_0|\} \setminus \bigcup_{i \in \mathbb N}X_i$. Finally, for every $Y \subseteq X$ with $|Y| = |X_0|$, we let $H(Y) = \{g \in G \mid gX_0 = Y\}$. Note that $H(Y)$ is generally not a subgroup of $G$, but a coset of the stabiliser $G_{X_0}$, or the empty set (in case there is no $g \in G$ which maps $X_0$ to $Y$).

\begin{lemma}
\label{lem:suborbitimages}
If $g,h \in H(Y)$, then $gS_n = hS_n$ for every $n \in \mathbb N$. In particular, if $Y \in \mathcal X_n$, then $H(Y) \subseteq G_n$.
\end{lemma}

\begin{proof}
For the first part, note that otherwise $g^{-1}h \in G_{X_0}$ would not fix $S_n$ setwise, contradicting the assumption that $S_n$ is an orbit of $G_{X_0}$. For the second part we observe that $Y \in \mathcal X_n$ if and only if there is $g \in G$ with $gX_0 = Y$ and $gS_i = S_i$ for every $i \geq n$. It follows from the first part that $hS_i = S_i$ for every $h \in H(Y)$ and every $i\geq n$, and therefore $H(Y) \subseteq G_n$.
\end{proof}

In what follows we will define a partition of the set $X$, and colour some parts of the partition to ensure that the colouring is asymmetric with respect to $G_n$ for $n \in \mathbb N$, and other parts to ensure that it is asymmetric with respect to $H(Y)$ for $Y \in \mathcal X_\infty$. By the above lemma this is enough to ensure that the resulting colouring is asymmetric with respect to $G$.

The following lemma will be used to construct the parts that are used for the groups $G_n$. It tells us that there are many different subsets of $X$ for which $G_n$ satisfies the condition of Theorem \ref{thm:babai}.

\begin{lemma}
\label{lem:compactsequence}
For each $n \in \mathbb N$ there is a sequence of pairwise disjoint, finite subsets $U_i \subseteq X$ such that the following properties are satisfied.
\begin{enumerate}
    \item $G_n$ setwise fixes each $U_i$.
    \item For every $g \in G_n \setminus \{id\}$ there is some $i \in \mathbb N$ such that $g$ non-trivially on $U_i$.
    \item $(G_n)_{(U_{i+1})} \leq (G_n)_{(U_{i})}$ for all $i \in \mathbb N$.
\end{enumerate}
\end{lemma}
\begin{proof}
We define the sequence $U_i$ by inductively defining values $k(i)$ for $i \in \mathbb N$ and setting 
\[
U_i = \bigcup_{k(i) < j \leq k(i+1)} S_j.
\]
Note that this implies that $G_n$ setwise fixes $U_i$ because $U_i$ is a union of orbits with respect to $G_n$. Moreover, a non-trivial element $g \in G_n$ which acts trivially on every $U_i$ would only permute elements of the finite set $X \setminus \bigcup_{j > k(1)}U_j$, and hence contradict the assumption that $G$ has infinite motion; in particular the first two properties will automatically be satisfied by sets constructed as above.

We start our inductive construction by setting $k(1) = n-1$ and $k(2) = n$, and thus $U_1 = S_n$. Now assume that $k(i)$ has already been defined. 

We claim that we can find some $K \in \mathbb N$ such that whenever $g \in G_n$ acts non-trivially on $U_i$, then it also acts non-trivially on at least one $S_j$ for $k(i) < j \leq K$. Assume for a contradiction that this is not possible. Then for every $K \in \mathbb N$ there is some $g_j \in G_n$ which acts non-trivially on $U_i$ but pointwise fixes $S_j$ for $k(i) < j \leq K$. Since the set $B := X \setminus \bigcup_{j\geq n} S_j$ is finite we may assume without loss of generality (by passing to a subsequence) that all $g_j$ agree on $B$. The permutation $g$ which agrees with the $g_j$ on $B$ and fixes all elements in $X \setminus B$ is an accumulation point of the sequence $g_j$. It follows that $g \in G$ because every $g_j$ is contained in $G$ and $G$ is closed. This is a contradicts the infinite motion of $G$ since $g$ only moves finitely many elements of $X$. This finishes the proof of the claim. 

To complete our inductive construction, we set $k(i+1) = K$ and note that this implies that any $g \in G_n$ which acts non-trivially on $U_i$ also acts non-trivially on $U_{i+1}$, and hence the sets $U_i$ also satisfy the third property claimed in the lemma.
\end{proof}

\begin{lemma}
\label{lem:infinitesubsequence}
Let $U_i$ be as in Lemma \ref{lem:compactsequence}, and let $I \subseteq \mathbb N$ be an infinite set. Then $G_n$ and $U_I = \bigcup_{i \in I} U_i$ satisfy the conditions of Observation \ref{obs:smallerset}, and thus $G_n$ can be seen as a permutation group on $U_I$. Moreover, there is an asymmetric $2$-colouring of $U_I$
\end{lemma}

\begin{proof}
Since $G_n$ fixes each individual $U_i$, it clearly also fixes $U_I$. Any $g \in G_n\setminus\{\id\}$ acts non-trivially on some $U_i$, and thus also on each $U_j$ for $j>i$. Since the set $I$ is infinite, this implies that $g$ acts non-trivially on $U_I$. For the `moreover' part note that the sequence $U_i, i \in I$ satisfies the conditions of Theorem \ref{thm:babai}.
\end{proof}

\begin{proof}[Proof of Theorem \ref{thm:main}]
For every $n \in \mathbb N$ we let $U_i(n)$ be the sequence obtained from Lemma~\ref{lem:compactsequence} for $G_n$. Pick a sequence $(s_k)_{k \in \mathbb N}$ of natural numbers such that every $n \in \mathbb N$ appears infinitely often in this sequence. Let $(Y_k)_{k \in \mathbb N}$ be an enumeration of $\mathcal X_\infty$.

Starting with $I_0(n) = \emptyset$ for every $n$, we inductively define finite sets $I_k(n) \subseteq \mathbb N$ for every $n \in \mathbb N$ and finite sets $Z_k \subseteq X$ as follows. Let 
\[
W_k = \bigcup_{i \in I_{k-1}(n)} U_i(n)\cup \bigcup_{i < k} Z_i.
\]
Let $g \in H_{Y_k}$. Since $Y_k \in \mathcal X_{\infty}$, there are infinitely many $i$ such that $gS_i \neq S_i$; note that $g S_i$ does not depend on the particular choice of $g$ by Lemma \ref{lem:suborbitimages}. We can pick some $i(k)$ such that $g S_{i(k)} \neq S_{i(k)}$ and $S_{i(k)} \cup g S_{i(k)}$ is disjoint from $W_k$. Set $Z_k = S_{i(k)} \cup g S_{i(k)}$.

Next we define the sets $I_k(n)$. For $n \neq s_k$ we set $I_{k}(n) = I_{k-1}(n)$. For $n = s_k$, there is some $i \in \mathbb N$ such that $U_i(n)$ is disjoint from $W_k \cup Z_k$. We set $I_k(n) = I_{k-1}(n) \cup \{i\}$.

Finally let $I(n) = \bigcup_{k \in \mathbb N} I_k(n)$ and define a colouring of $X$ as follows. For $Z_k$ we colour $S_{i(k)}$ with colour $0$ and $Z_k \setminus S_{i(k)}$ with colour $1$. Colour $U_{I(n)} = \bigcup_{i \in I(n)} U_i$ by an asymmetric $2$-colouring with respect to $G_n$; this is possible by Lemma \ref{lem:infinitesubsequence} because $n$ appears infinitely often in the sequence $s_k$ and thus $I(n)$ is infinite. If there are any remaining uncoloured elements of $X$, then colour them arbitrarily. 

Assume now that there is some $g \in G$ which preserves the resulting colouring. Observe that $g \notin G_n$ because of the colouring on $U_{I(n)}$. Hence $gX_0 = Y_k$ for some $k \in \mathbb N$. But then (by Lemma \ref{lem:suborbitimages}) $g$ maps some element of $S_{i(k)}$ to an element of $Z_k \setminus S_{i(k)}$. The former is coloured with colour $0$, the latter is coloured with colour $1$. This contradicts the assumption that $g$ was colour preserving.
\end{proof}

\bibliographystyle{plain}
\bibliography{bibliography.bib}
\end{document}